\documentclass[10pt,a4paper]{amsart}

\usepackage{amsmath,amsfonts,amscd,amssymb,amsthm,bm,amsrefs}
\usepackage{verbatim}
\usepackage{color}
\usepackage{hyperref}
\usepackage{mathrsfs}

\theoremstyle{theorem}
\newtheorem{lemma}{Lemma}[section]
\newtheorem{theorem}[lemma]{Theorem}
\newtheorem{corollary}[lemma]{Corollary}
\newtheorem{proposition}[lemma]{Proposition}
\newtheorem{question}{Question}

\theoremstyle{definition}
\newtheorem{definition}[lemma]{Definition}

\theoremstyle{remark}
\newtheorem{remark}[lemma]{Remark}

\newcommand{\link}{\mathrm{link}}
\newcommand{\R}{\mathbb{R}}
\newcommand{\C}{\mathbb{C}}
\newcommand{\Z}{\mathbb{Z}}
\newcommand{\D}{\mathbb{D}}
\newcommand{\N}{\mathbb{N}}

\newcommand{\wind}{\operatorname{wind}}

\newcommand{\jtil}{\widetilde{J}}
\newcommand{\util}{\widetilde{u}}
\newcommand{\vtil}{\widetilde{v}}
\newcommand{\wtil}{\widetilde{w}}
\newcommand{\CZ}{{\rm CZ}}
\newcommand{\jbar}{\bar{J}}
\renewcommand{\P}{\mathcal{P}}
\newcommand{\Mfast}{\mathcal{M}^{\rm fast}}
\newcommand{\ev}{{\rm ev}}

\newcommand{\Pscr}{\mathscr{P}}

\begin{document}

\title[]{Global surfaces of section with positive genus for 
dynamically convex 
Reeb flows}
\author[]{Umberto L. Hryniewicz}
\author[]{Pedro A. S. Salom\~ao}
\author[]{Richard Siefring}
\date{\today}

\address{Umberto L. Hryniewicz -- RWTH Aachen, Jakobstrasse 2, Aachen 52064, Germany}
\email{hryniewicz@mathga.rwth-aachen.de}

\address{Pedro A. S. Salom\~ao -- Instituto de Matem\'atica e Estat\'istica, Departamento de Matem\'atica, Universidade de S\~ao Paulo, Rua do Mat\~ao, 1010, Cidade Universit\'aria, S\~ao Paulo SP, Brazil 05508-090 \newline\newline NYU-ECNU Institute of Mathematical Sciences at NYU Shanghai, 3663 Zhongshan Road North, Shanghai, 200062, China}
\email{psalomao@ime.usp.br; pas383@nyu.edu}

\address{Richard Siefring -- Fakult\"at f\"ur Mathematik \\ Ruhr-Universit\"at Bochum \\ 44780 Bochum \\ Germany}
\email{richard.siefring@ruhr-uni-bochum.de}

\begin{abstract}
%We introduce a symplectic capacity defined in terms of symplectic areas of global surfaces of section in order to study uniqueness of capacities on convex bodies in dimension four. 
We establish some new existence results for global surfaces of section of dynamically convex Reeb flows on the three-sphere. These sections often have genus, and are the result of a combination of pseudo-holomorphic methods with some elementary ergodic methods.
\end{abstract}

\maketitle

%\tableofcontents

%\pagebreak

\begin{center}
{\it Dedicated to Prof. Claude Viterbo on the occasion of his 60th birthday.}
\end{center}

%\bigskip
%\bigskip

\section{Introduction and main results}

Let $(z_0,z_1)$ be complex coordinates in $\C^2$, $S^3 = \{|z_0|^2+|z_1|^2=1\}$, and $\alpha_0$ be the standard Liouville form $(-i/4) \ \Sigma_{j} \ \bar z_j dz_j - z_jd\bar z_j$. The standard symplectic form on $\C^2$ is $\omega_0 = d\alpha_0$. The fibres of the Hopf fibration are the periodic Reeb orbits of the contact form $\lambda_0$ on $S^3$ induced by $\alpha_0$. Let us call the Reeb flow of $\lambda_0$ the \textit{Hopf flow}. The contact structure $\xi_0 = \ker\lambda_0$ is called {\it standard}.

The contact form $\lambda_0$ is the first example of a \textit{dynamically convex} contact form. In $S^3$ a contact form $\lambda$ is said to be dynamically convex if all periodic orbits have Conley-Zehnder index $\geq 3$ when computed in a global $d\lambda$-symplectic frame of 
%the contact structure 
$\ker\lambda$. This notion was introduced by Hofer, Wysocki and Zehnder (HWZ) in~\cite{convex}.

One can show quite explicitly that all finite collections of periodic orbits of the Hopf flow span some global surface of section, see~\cite{AGZ}. It is natural to ask if this property remains true for all dynamically convex Reeb flows on $S^3$, in particular for all strictly convex energy levels in $(\C^2,\omega_0)$ (\cite[Theorem~3.4]{convex}). This might be too ambitious to try to prove, and one may be led to naively think that it is easy to find a counter-example. There is, however, another natural way to generalise this property of the Hopf flow to the Reeb flows of all dynamically convex contact forms on $S^3$. Since the Hopf fibres are unknotted with self-linking number $-1$, one might ask if all finite collections of periodic Reeb orbits of this kind span some global surface of section. This is our first result.

\begin{theorem}\label{main3}
Let $L$ be any link formed by periodic Reeb orbits of a dynamically convex contact form on $S^3$ whose components are unknotted with self-linking number~$-1$. Then $L$ bounds a global surface of section for the Reeb flow.
\end{theorem}

There are no hidden genericity assumptions on the contact form. The genus of these sections will typically explode with the number of boundary orbits.  Moreover, there is no need to specify the contact structure since only the standard one can be defined by a dynamically convex contact form on $S^3$, see~\cite{char2}.

%A proof relying exclusively on pseudo-holomorphic curve methods would have to handle the not easy task of getting compactness for curves with genus. 

A proof relying exclusively on pseudo-holomorphic curves would be complicated by the fact, originally observed in~\cite{hofer_survey}, that transversality fails for curves with genus which are everywhere transverse to the flow.  The solution proposed to this problem in~\cite{hofer_survey} is to consider a perturbation of the holomorphic curve equation which corrects the transversality problem, but seriously complicates the compactness theory (see~\cite{abbas,doicu_fuchs_1,doicu_fuchs_2}). However, dealing with genus is unavoidable since the links covered by Theorem~\ref{main3} typically have positive Seifert genus. A proof without pseudo-holomorphic curves seems out of reach since dynamical convexity is an assumption only on the periodic orbits, and holomorphic curve techniques have proven to be one of the very few -- if not the only -- effective methods for finding surfaces of section under assumptions of this kind.

Let us outline the argument. The main step is the result from~\cite{hryn_jsg} stating that every component of a link $L$ as in Theorem~\ref{main3} bounds a disk-like global surface of section. At this point ergodic methods come to aid via asymptotic cycles. We use the statement from~\cite{SFS} refining a celebrated result due to Fried~\cite{fried}. The disks can be used to check the hypotheses of~\cite[Theorem~1.3]{SFS}. Each disk has uniformly bounded return time, hence all invariant measures in $S^3\setminus L$ positively hit the sum of the cohomology classes dual to each disk. Finally, positivity of rotation numbers follows from dynamical convexity.

\begin{remark}\label{rmk_CDR}
It was explained to us by Colin, Dehornoy and Rechtman that the input from pseudo-holomorphic curves from~\cite{hryn_jsg} can be used in a more elementary way, avoiding asymptotic cycles. One can take the union of the disks and ``resolve intersections'' to construct the desired sections. This idea is extensively used in~\cite{CDR}.
\end{remark}

In~\cite{ghys} Ghys introduced the notion of \textit{right (left) handed} vector field on a homology three-sphere, and explained that all finite collections of periodic orbits of such a vector field span a global surface of section. The Hopf flow is the simplest example of a right handed vector field. Examples of left handed geodesic flows on negatively curved two-dimensional orbifolds are presented by Dehornoy~\cite{dehornoy}. Right handedness provides deep insight on the dynamics. For instance, it follows that every finite collection of periodic orbits is a fibered link, hence there are strong knot theoretical restrictions. Moreover, as soon as such a collection is ``misplaced'' then Nielsen-Thurston theory might be used to obtain entropy via the study of the isotopy class of the return map. %We believe it to be an interesting task that of understanding how large the class right handed Reeb flows is. 

\begin{question}\label{ques_right_handed}
Is the Reeb flow of every dynamically convex contact form on $S^3$ right handed?
\end{question}

A positive answer is probably very hard to obtain, even in finite-dimensional families of interesting flows such as those appearing in Celestial Mechanics. One is then tempted to look for examples to give a negative answer, but they might not exist. In the context of the $3$-body problem, we refer to~\cite{lhuissier} for a discussion of a version of this question, and to the book~\cite{bookFK} by Frauenfelder and van Koert for a discussion on global surfaces of section, including a related conjecture of Birkhoff. The existence of genus zero global surfaces of section with prescribed binding orbits has been clarified in~\cite{HSW}. In an upcoming paper~\cite{FH} it will be shown that geodesic flows on $S^2$ with curvatures pinched by some explicit constant lift to right handed Reeb flows on $S^3$.

\begin{remark}
The dynamical convexity assumption is essential in Question~\ref{ques_right_handed}, as one can easily check. It is, however, more subtle to rule out specific types of global surfaces of sections when dynamical convexity is dropped. For instance, in~\cite{O} one finds examples of contact forms on $S^3$ without disk-like global surfaces of section. The situation in higher dimensions is still wide open, but in~\cite{MO} there are interesting new constructions for the spatial circular restricted $3$-body problem.
\end{remark}

Our second result is closely connected to the following question.

\begin{question}[HWZ~\cite{convex}]\label{ques_HWZ}
Is the minimal period among closed Reeb orbits of a dynamically convex contact form on $S^3$ equal to the contact area of some disk-like global surface of section ?
\end{question}

Starting from a nondegenerate dynamically convex contact form on $S^3$, Hutchings and Nelson~\cite{HN} were able to implement the construction of the chain complex of Cylindrical Contact Homology (CCH), originally explained by Eliashberg, Givental and Hofer in~\cite{EGH}. The arguments from~\cite{HN} rely on elementary pseudo-holomorphic curve methods. Invariance of the resulting homology is delicate and requires sophisticated technology, for instance, one can use the Polyfold Theory introduced by Hofer, Wysocki and Zehnder; see~\cite{polyfolds_survey} for a survey. It will be shown in~\cite{HHR} that elementary methods are still enough to get invariance of CCH in its lowest degree. This is enough to get the first spectral invariant $c_1^{\rm CCH}$ well-defined. Also in~\cite{HHR} it will be shown that $c_1^{\rm CCH}$ is the action of some periodic orbit with Conley-Zehnder index~$3$ realised as the asymptotic limit of a pseudo-holomorphic plane. 
%In addition, a result by Abbondandolo and Kang~\cite{AK} implies that $c_1^{\rm CCH}$ is equal to the minimal action when the contact form comes from a smooth convex domain with strictly convex boundary. 
Hence we get the following consequence of a combination of Corollary~\ref{cor_main4} below with some of the results from~\cite{HHR}: {\it ``The spectral invariant $c_1^{\rm CCH}$ of a nondegenerate dynamically convex contact form on $S^3$ is the contact area of some global surface of section.'' %Moreover, if the contact form is induced by a smooth convex domain with strictly convex boundary then the minimal period is the area of a global surface of section.
} %This will be the confirmation of a phenomenon similar to a positive answer to Question~\ref{ques_HWZ} in the nondegenerate case, where the minimal period is replaced by $c_1^{\rm CCH}$ and the global surface of section might have genus. %The convexity assumption and the results from~\cite{AK} allow us to get $c_1^{\rm CCH}$ equal to the minimal action.

\begin{theorem}\label{main4}
Let $\lambda$ be a contact form on $S^3$ that is both non-degenerate and dynamically convex up to action $C$. Suppose that a periodic Reeb orbit $P=(x,T)$ satisfies $T\leq C$ and is the asymptotic limit of a fast finite-energy plane. Then the knot $x(\R)$ spans a global surface of section for the Reeb flow.
\end{theorem}

In the theorem above and the corollary below a periodic Reeb orbit is a pair $P=(x,T)$ where $x$ is a periodic trajectory of the Reeb flow and $T>0$ is a period, not necessarily the primitive one.

\begin{corollary}\label{cor_main4}
Let $\lambda$ be a contact form on $S^3$ that is both nondegenerate and dynamically convex up to action $C$. Suppose that a periodic Reeb orbit $P=(x,T)$ satisfies $T\leq C$, $\CZ(P)=3$, and is the asymptotic limit of a finite-energy plane. Then the knot $x(\R)$ spans a global surface of section for the Reeb flow.
\end{corollary}

\begin{proof}
The equality $\CZ(P)=3$ implies that any finite-energy plane asymptotic to $P$ is fast. Now apply Theorem~\ref{main4}.
\end{proof}

The proof of Theorem~\ref{main4} is based on a certain class of pseudo-holomorphic planes called {\it fast}, but also uses ergodic methods (asymptotic cycles)~\cite{fried,ghys,SFS,schwartzman,sullivan}. Fast planes were originally introduced in~\cite{tese} and later used in~\cite{fast,hryn_jsg,HS_D,HLiS,elliptic} to prove several existence results on global surfaces of section. Roughly speaking, an end of a plane is in some sense a gradient trajectory of the action functional, and the results from~\cite{props1} basically say that the approach to the periodic orbit is governed by an eigenvector of an operator that plays the role of the Hessian of the action, the so-called {\it asymptotic operator}. The term ``fast'' refers to the fact that the eigenvalue of this asymptotic eigenvector has the same winding number of the most negative eigenvalue allowed, hence the approach is roughly the fastest it can be; see Definition~\ref{def_fast_plane}.

\begin{remark}
The global sections obtained from Theorem~\ref{main4} and Corollary~\ref{cor_main4} may have genus. Note that $P$ is not assumed to be simply covered, but still the global sections obtained are Seifert surfaces for the knot $x(\R)$. %This is the first time when fast planes with self-intersections are used.
\end{remark}

\noindent \textbf{Acknowledgments.} UH would like to thank V. Colin, P. Dehornoy and A. Rechtman for enlightening discussions leading to Remark~\ref{rmk_CDR} during the {\it 104e rencontre entre math\'ematiciens et physiciens th\'eoriciens} held at IRMA (2019), the organisers of this event for the invitation, and IRMA for the hospitality. UH also thanks Kai Cieliebak and Urs Frauenfelder for helpful conversations, and the hospitality during a visit to Universit\"at Augsburg. RS thanks Jungsoo Kang and Urs Frauenfelder for helpful discussions. We thank the referee for interesting and helpful feedback. RS was partially supported by the SFB/TRR 191 ``Symplectic Structures in Geometry, Algebra and Dynamics'' funded by the DFG (Projektnummer 281071066 – TRR 191). PS is partially supported by FAPESP 2016/25053-8 and CNPq 306106/2016-7. PS acknowledges the support of NYU-ECNU Institute of Mathematical Sciences at NYU Shanghai.

\section{Preliminaries}

Let $\lambda$ be a contact form on a $3$-manifold $M$. The contact structure is denoted by $\xi = \ker\lambda$.

\subsection{Periodic orbits, asymptotic operators and Conley-Zehnder indices}

The Reeb vector field $X_\lambda$ of $\lambda$ is implicitly  defined by $$ i_{X_\lambda}d\lambda=0, \qquad\qquad i_{X_\lambda}\lambda=1. $$ Its flow $\phi^t$ is called the Reeb flow. Let us fix a marked point on every periodic trajectory of $\phi^t$. A \textit{periodic Reeb orbit} is a pair $P=(x,T)$ where $x:\R\to M$ is a periodic trajectory of $\phi^t$ such that $x(0)$ is the marked point, and $T>0$ is a period. It is not required that $T$ is the primitive period. The set of periodic orbits will be denoted by $\P(\lambda)$. If $T_0>0$ is the primitive period of $x$ then $k = T/T_0 \in \N$ is called the \textit{covering multiplicity} of~$P$. The contact form $\lambda$ is said to be \textit{nondegenerate up to action $C \in (0,+\infty]$} if $1$ is not in the spectrum of $d\phi^T|_{x(0)} : {\xi}|_{x(0)} \to {\xi}|_{x(0)}$, for all $P=(x,T) \in \P(\lambda)$ such that $T\leq C$. When $C=+\infty$ we simply say that $\lambda$ is {\it nondegenerate}.

There is an unbounded operator on $L^2(x(T\cdot)^*\xi)$ 
\begin{equation*}
\eta \mapsto J(-\nabla_t\eta+T\nabla_\eta X_\lambda)
\end{equation*}
associated to a pair $(P,J)$, where $P = (x,T) \in \P$ and $J:\xi\to\xi$ is a $d\lambda$-compatible complex structure. Here $\nabla$ is a symmetric connection on $TM$ and $\nabla_t$ denotes the associated covariant derivative along the loop $t \in \R/\Z \mapsto x(Tt)$. This is called the \textit{asymptotic operator}. It does not depend on the choice of $\nabla$. It is self-adjoint when $L^2(x(T\cdot)^*\xi)$ is equipped with the inner-product
\begin{equation*}
(\eta,\zeta) \mapsto \int_{\R/\Z} d\lambda(x(Tt))(\eta(t),J(x(Tt))\zeta(t)) \ dt
\end{equation*}
Its spectrum is discrete, consists of eigenvalues whose geometric and algebraic multiplicities coincide, and accumulates at $\pm\infty$. It turns out that $\lambda$ is nondegenerate if, and only if, $0$ is never an eigenvalue of an asymptotic operator. The eigenvectors are nowhere vanishing sections of $x(T\cdot)^*\xi$ since they solve linear ODEs. Hence they have well-defined winding numbers with respect to a $d\lambda$-symplectic trivialisation $\sigma$ of $x(T\cdot)^*\xi$. The winding number is independent of the choice of eigenvector of a given eigenvalue. This allows us to talk about the winding number $$ \wind_\sigma(\nu) \in \Z $$ of an eigenvalue $\nu$ with respect to $\sigma$. For every $k\in\Z$ there are precisely two eigenvalues satisfying $\wind_\sigma = k$, multiplicities counted and, moreover, $\nu_1\leq\nu_2 \Rightarrow \wind_\sigma(\nu_1)\leq\wind_\sigma(\nu_2)$. These properties are independent of $\sigma$. These properties of the asymptotic operator have been established in~\cite{props2}. Given any $\delta \in \R$ we set
\begin{equation*}
\begin{aligned}
\alpha^{<\delta}_\sigma(P) &= \max \ \{ \wind_\sigma(\nu) \mid \nu \ \text{eigenvalue}, \ \nu<\delta \} \\
\alpha^{\geq\delta}_\sigma(P) &= \min \ \{ \wind_\sigma(\nu) \mid \nu \ \text{eigenvalue}, \ \nu\geq\delta \} \\
p^{\delta}(P) &= \alpha^{\geq\delta}_\sigma(P) - \alpha^{<\delta}_\sigma(P)
\end{aligned}
\end{equation*}
Finally we consider the constrained Conley-Zehnder index
\begin{equation}
\CZ^\delta_\sigma(P) = 2 \alpha^{<\delta}_\sigma(P) + p^{\delta}(P)
\end{equation}
Note that this is defined also in degenerate situations.

A contact form $\lambda$ is {\it dynamically convex up to action $C\in (0,+\infty]$} if $c_1(\xi,d\lambda)$ vanishes on $\pi_2 \hookrightarrow H_2$, and every contractible\footnote{This means that the loop $t\in\R/\Z \mapsto x(Tt)$ is contractible.} $P=(x,T)\in\P(\lambda)$ satisfying $T\leq C$ also satisfies $\CZ^0_{\sigma_{\rm disk}}(P) \geq 3$. Here $\sigma_{\rm disk}$ is a trivialisation that extends to a capping disk. If $C=+\infty$ then we say that $\lambda$ is {\it dynamically convex}.

\begin{remark}
Dynamical convexity was introduced by HWZ in~\cite{convex}. The assumption that $c_1(\xi,d\lambda)$ vanishes on spheres implies that the homotopy class of $\sigma_{\rm disk}$ does not depend on the choice of a capping disk.
\end{remark}

\subsection{Pseudo-holomorphic curves in symplectisations}\label{ssec_pseudo_hol_curves}

From now on we assume $M$ is closed. Let $J$ be a compatible complex structure on the symplectic vector bundle~$(\xi,d\lambda)$. Hofer~\cite{93} considers an almost complex structure $\jtil$ defined on $\R\times M$ by
\begin{equation}\label{R_inv_J}
\jtil : \partial_a \mapsto X_\lambda \qquad\qquad \jtil|_{\xi} = J
\end{equation}
where $X_\lambda$ and $\xi$ are seen as $\R$-invariant objects in $\R\times M$. Then $\jtil$ is $\R$-invariant. Consider a closed Riemann surface $(S,j)$, a finite set $\Gamma \subset S$ and a pseudo holomorphic map $$ \util = (a,u) : (S\setminus\Gamma,j) \to (\R\times M,\jtil) $$ satisfying a finite-energy condition $$ 0< E(\util) = \sup_{\phi} \int_{S\setminus\Gamma} \util^*d(\phi\lambda) < \infty $$ where the supremum is taken over the set of $\phi:\R \to [0,1]$ satisfying $\phi'\geq0$. The number $E(\util)$ is called the \textit{Hofer energy}. Such a map is called a \textit{finite-energy} map. Points in $\Gamma$ are called \textit{punctures}. A puncture $z\in\Gamma$ is \textit{positive} or \textit{negative} if $a(w)\to+\infty$ or $a(w)\to-\infty$ when $w\to z$, respectively. It is called \textit{removable} if $\limsup |a(w)| < \infty$ when $w\to z$. It turns out that every puncture is positive, negative or removable, and that $\util$ can be smoothly extended across a removable puncture; see~\cite{93}.

Let $z\in\Gamma$ and let $K$ be a conformal disk centred at $z$, i.e. there is a biholomorphism $\varphi : (K,j,z) \to (\D,i,0)$. Then $K\setminus\{z\}$ admits positive holomorphic polar coordinates $(s,t) \in [0,+\infty) \times \R/\Z$ defined by $(s,t) \simeq \varphi^{-1}(e^{-2\pi(s+it)})$, and negative holomorphic polar coordinates $(s,t) \in (-\infty,0] \times \R/\Z$ defined by $(s,t) \simeq \varphi^{-1}(e^{2\pi(s+it)})$.

\begin{theorem}[Hofer~\cite{93}]
Let $z\in\Gamma$ be a non-removable puncture, and $(s,t)$ be positive holomorphic polar coordinates at $z$. For every sequence $s_n \to +\infty$ there exist a subsequence $s_{n_j}$ and $P=(x,T) \in \P$ such that $u(s_{n_j},t) \to x(\epsilon Tt+d)$ in $C^\infty(\R/\Z,M)$, for some $d\in\R$, where $\epsilon = \pm1$ is the sign of the puncture.
\end{theorem}

From now on we denote by 
\begin{equation}\label{proj_ctct_str}
\pi_\lambda : TM \to \xi
\end{equation}
the projection along $X_\lambda$.

%The following is~\cite[Definition~3.11]{fast}.

%\begin{definition}[Definition~3.11 from~\cite{fast}]\label{def_non_deg_punc}
%Let $z$ be a non-removable puncture of $\util=(a,u)$ of sign $\epsilon = \pm1$, and let $(s,t)$ denote positive holomorphic polar coordinates at $z$. Denote by $S(z)$ the connected component of the domain of $\util$ containing $z$. Then $z$ is called a \textit{non-degenerate puncture} if there exists $P=(x,T) \in \P$ such that:
%\begin{itemize}
%\item There exists $c\in\R$ such that $a(s,t)-\epsilon Ts-c \to 0$ in $C^0(\R/\Z)$ as $s\to+\infty$
%\item There exists $d\in\R$ such that $u(s,t) \to x(\epsilon Tt+d)$ in $C^0(\R/\Z)$ as $s\to+\infty$.
%\item If $\int_{S(z)} u^*d\lambda > 0$ then $\pi_\lambda \circ du(s,t) \neq 0$ if $s$ is large enough.
%\item If $u(s,t) = \exp_{x(\epsilon Tt+d)}\zeta(s,t)$, with $\zeta(s,t) \in \xi_0|_{x(\epsilon Tt+d)}$, then there exists $r>0$ such that $e^{rs}|\zeta(s,t)| \to 0$ in $C^0(\R/\Z)$ as $s\to+\infty$.
%\end{itemize}
%Here $\exp$ and $|\cdot|$ are induced by any Riemannian metric on $S^3$. The orbit $P$ is called the asymptotic limit of $\util$ at $z$.
%\end{definition}

%Some of the results from~\cite{props1} are contained in the following statement.

\begin{theorem}[HWZ~\cite{props1}]
Suppose that $\lambda$ is non-degenerate up to action $C$, and that $z$ is a non-removable puncture of a finite-energy curve $\util=(a,u)$ in $(\R\times M,\jtil)$ with Hofer energy $E(\util)\leq C$. Let $(s,t)$ be positive holomorphic polar coordinates at~$z$. There exist $P=(x,T) \in \mathcal{P}$, $d\in\R$ such that $u(s,t) \to x(\epsilon Tt+d)$ in $C^\infty(\R/\Z,M)$ as $s\to+\infty$, where $\epsilon = \pm1$ is the sign of the puncture. 
\end{theorem}

\begin{remark}
The orbit $P$ is called the \textit{asymptotic limit} of $\util$ at $z$. 
\end{remark}

Consider the space $\R/\Z\times\C$ equipped with coordinates $(\vartheta,z=x_1+ix_2)$ and contact form $\beta_0 = d\vartheta + x_1dx_2$.

\begin{definition}
A Martinet tube for $P = (x,T) \in \P$ is a smooth diffeomorphism $\Psi:\mathcal{N} \to\R/\Z\times \D$ defined on a smooth compact neighborhood $\mathcal{N}$ of $x(\R)$ such that: 
\begin{itemize}
\item $\Psi(x(T\vartheta/k)) = (\vartheta,0)$ for all $\vartheta \in \R/\Z$, where $k\in\N$ is the covering multiplicity of $P$.
\item $\lambda|_{\mathcal{N}} = \Psi^*(g\beta_0)$, where $g : \R/\Z\times \D \to (0,+\infty)$ is smooth and satisfies $g(\vartheta,0) = T/k$, $dg(\vartheta,0) = 0$ for all $\vartheta \in \R/\Z$.
\end{itemize}
\end{definition}

\begin{theorem}[HWZ~\cite{props1}, Mora-Donato~\cite{Mora}, Siefring~\cite{sie_CPAM}]\label{thm_asymptotics}
Suppose that $\lambda$ is non-degenerate up to action $C>0$, and that $z$ is a non-removable puncture of sign $\epsilon = \pm1$ of a finite-energy curve $\util=(a,u)$ with Hofer energy $E(\util)\leq C$. Let $(s,t)$ be positive or negative holomorphic polar coordinates at $z$ when $\epsilon=+1$ or $\epsilon=-1$, respectively. Consider any Martinet tube $\Psi:\mathcal{N} \to\R/\Z\times \D$ for the asymptotic limit $P$ of $\util$ at~$z$, and $s_0\gg1$ such that $|s|\geq s_0 \Rightarrow u(s,t) \in \mathcal{N}$. Write $\Psi(u(s,t)) = (\vartheta(s,t),z(s,t))$ for $|s|\geq s_0$. Up to a rotation, we can assume $u(s,0) \to x(0)$ as $\epsilon s\to+\infty$.

If $z(s,t)$ does not vanish identically then the following holds. There exist $r>0$ and an eigenvalue $\nu$ of the asymptotic operator of $(P,J)$ satisfying $\epsilon\nu<0$, such that:
\begin{itemize}
\item There exist $c,d\in\R$ and a lift $\tilde\vartheta:\R\times\R \to \R$ of $\vartheta(s,t)$ such that $$ \lim_{\epsilon s\to+\infty} \sup_{t\in\R/\Z} e^{r\epsilon s} \left( |D^\beta[a(s,t)-Ts-c]| + |D^\beta[\tilde\vartheta(s,t)-kt-d]| \right) = 0  $$ holds for every partial derivative $D^\beta = \partial^{\beta_1}_s\partial^{\beta_2}_t$, where $k$ is the covering multiplicity of $P$.
\item There exists an eigenvector of $\nu$, represented as a nowhere vanishing vector field $v(t)$ in the frame $\{\partial_{x_1},\partial_{x_2}\}$ along $P$, such that
\begin{equation*}
z(s,t) = e^{\nu s} (v(t)+R(s,t))
\end{equation*}
for some $R(s,t)$ satisfying $|D^\beta R(s,t)| \to 0$ in $C^0(\R/\Z)$ as $\epsilon s\to+\infty$, for every partial derivative $D^\beta = \partial^{\beta_1}_s\partial^{\beta_2}_t$.
\end{itemize}
\end{theorem}

The alternative $z(s,t) \equiv 0$ can be expressed independently of coordinates as saying that the end of the domain of $\util$ corresponding to the puncture is mapped into the trivial cylinder over the asymptotic limit. In this case we say that $\util$ has \textit{trivial} asymptotic behaviour at the puncture. Otherwise, the asymptotic behaviour is said to be \textit{nontrivial} at the puncture.

\begin{remark}\label{rmk_asymp_evalue}
The eigenvalue $\nu$ provided by Theorem~\ref{thm_asymptotics} is called the asymptotic eigenvalue of $\util$ at the puncture $z$.
\end{remark}

Let us recall some of the invariants introduced in~\cite{props2} in the $\R$-invariant case. Let $\util=(a,u)$ be a finite-energy curve on $(\R\times M,\jtil)$, defined on a connected domain. Assume that $\lambda$ is nondegenerate up to action $E(\util)$. It can be shown that if $\pi_\lambda \circ du$ does not vanish identically then its zeros are isolated and count positively. Theorem~\ref{thm_asymptotics} further implies that there are finitely many zeros in this case. HWZ~\cite{props2} define
\begin{equation}\label{def_wind_pi}
\wind_\pi(\util) \geq 0
\end{equation}
to be the algebraic count of zeros in case $\pi_\lambda \circ du$ does not vanish identically. Fix a $d\lambda$-symplectic trivialisation $\sigma$ of $u^*\xi$. Let $z$ be a puncture of $\util$ with asymptotic limit $P=(x,T)$. The asymptotic behaviour described in Theorem~\ref{thm_asymptotics} allows one to deform $\sigma$ so that it extends to a trivialisation of $x(T\cdot)^*\xi$. Let $\wind_\infty(\util,z,\sigma) \in \Z$ be defined to be the winding of the asymptotic eigenvalue of $\util$ at $z$ with respect to the extension of $\sigma$ to $x(T\cdot)^*\xi$. Finally we consider $$ \wind_\infty(\util) = {\sum}_+ \wind_\infty(\util,z,\sigma) - {\sum}_-  \wind_\infty(\util,z,\sigma) $$ where $\Sigma_+$ denotes a sum over the positive punctures, and $\Sigma_-$ is a sum over the negative punctures. Standard degree theory shows that 
\begin{equation}\label{identity_winds}
\wind_\pi(\util) = \wind_\infty(\util) - \chi + \#\{\text{punctures}\}
\end{equation}
holds provided $\int u^*d\lambda > 0$. Note that $\wind_\infty(\util)$ does not depend on the choice of trivialisation $\sigma$ of $u^*\xi$.

Denote by $(\overline{\C} = \C \cup \{\infty\},i)$ the Riemann sphere. For the next two definitions consider a finite-energy plane $\util = (a,u) : (\C,i) \to (\R\times M,\jtil)$ and assume that $\lambda$ is nondegenerate up to action $E(\util)$. By Stokes theorem, $\infty$ must be a positive puncture, and the similarity principle implies that $\int_\C u^*d\lambda>0$.

\begin{definition}\label{def_fast_plane}
The plane $\util$ is said to be \textit{fast} if $\wind_\infty(\util)=1$.
\end{definition}

\begin{definition}\label{def_cov_plane}
The covering multiplicity ${\rm cov}(\util)$ of the plane $\util$ is the covering multiplicity of its asymptotic limit.
\end{definition}

Fast planes in symplectisations were originally introduced in~\cite{fast}.

\begin{lemma}\label{lemma_fast_symplectisations}
If $\util=(a,u)$ is a fast plane then $\util$ is somewhere injective and the map $u:\C\to M$ is an immersion transverse to $X_\lambda$.
\end{lemma}

\begin{proof}
That $u$ is an immersion transverse to $X_\lambda$ follows from~\eqref{def_wind_pi} and~\eqref{identity_winds}. If $\util$ is not somewhere injective then it covers another plane via a polynomial map of degree~$\geq2$, but this forces $\util$ to have critical points, in contradiction to $u$ being an immersion; here we used that the Cauchy-Riemann equations force a critical point to be a zero of the derivative of $\util$.
\end{proof}

\subsection{Asymptotic cycles}\label{ssec_asymp_cycles}

Here we explain the basics on asymptotic cycles, and state the main result from~\cite{SFS}. Let $\phi^t$ be a smooth flow on a smooth, closed, oriented and connected $3$-manifold $M$, and let $L$ be a link consisting of (non-constant) periodic orbits. The set of $\phi^t$-invariant Borel probability measures on $M\setminus L$ is denoted by $\Pscr_\phi(M\setminus L)$. Fix an auxiliary Riemannian metric $g$ on $M$. If $p \in M\setminus L$ is recurrent and the sequence $T_n \to +\infty$ satisfies $\phi^{T_n}(p) \to p$, then we denote by $k(T_n,p)$ loops obtained by concatenating to $\phi^{[0,T_n]}(p)$ a $g$-shortest path from $\phi^{T_n}(p)$ to $p$. With $\mu \in \Pscr_\phi(M\setminus L)$ and $y\in H^1(M\setminus L;\R)$ fixed, one can use the Ergodic Theorem to show that $\mu$-almost all points $p \in M\setminus L$ have the following properties: $p$ is recurrent, and the limits
\begin{equation*}
\lim_{n\to+\infty} \frac{\left<y,k(T_n,p)\right>}{T_n}
\end{equation*}
exist independently of $T_n$ and $g$, and define a $\mu$-integrable function $f_{\mu,y}$. The integral 
\begin{equation*}
\mu \cdot y := \int_{M\setminus L} f_{\mu,y} \ d\mu
\end{equation*}
is, by definition, the \textit{intersection number} of $\mu$ and $y$.

If $\gamma$ is the periodic orbit given by a connected component of $L$, then $\xi_\gamma = TM|_\gamma/T\gamma$ is a rank-$2$ vector bundle over $\gamma$. It carries an orientation induced by the ambient orientation and the flow orientation on $\gamma$. A positive frame of $\xi_\gamma$ allows one to identify $\xi_\gamma \simeq \gamma \times \C \simeq \R/T_\gamma\Z \times \C$, where $T_\gamma>0$ is the primitive period. If $t$ is the coordinate on $\R/T_\gamma\Z$ (given by the flow) and $\theta \in \R/2\pi\Z$ is the polar angle on $\C^*$ then $\{dt,d\theta\}$ is a basis of $H^1((\xi_\gamma\setminus0)/\R_+;\R)$. With the aid of any exponential map the class $y$ induces a class in this homology group that can be written as $pdt+qd\theta$. The coefficients $p,q\in\R$ depend only on $y$ and on the chosen frame. If $u$ is a nonzero vector in $\xi_\gamma$ then using the frame  we can write $d\phi^t \cdot u \simeq r(t)e^{i\theta(t)}$ with smooth functions $r(t)>0,\theta(t)$. The rotation number
\begin{equation}\label{def_rot_number}
\rho^y(\gamma) :=  \frac{T_\gamma}{2\pi} \left( p + q \lim_{t\to+\infty} \frac{\theta(t)}{t} \right)
\end{equation}
is independent of the choice of frame, and of the vector $u$.

The following statement is a refinement of a result due to Fried~\cite{fried}, see also Sullivan~\cite{sullivan}.

\begin{theorem}[\cite{SFS}]\label{thm_SFS}
Let $b \in H_2(M,L;\Z)$ be induced by an oriented Seifert surface with boundary $L$, and denote by $b^* \in H^1(M\setminus L;\R)$ the class dual to $b$. Consider the following assertions:
\begin{itemize}
\item[(i)] $L$ bounds a global surface of section representing $b$.
\item[(ii)] $L$ binds an open book decomposition whose pages are global surfaces of section representing $b$.
\item[(iii)] The following hold:
\begin{itemize}
\item[(a)] $\rho^{b^*}(\gamma)>0$ for every connected component $\gamma \subset L$.
\item[(b)] $\mu\cdot b^*>0$ for every $\mu \in \Pscr_\phi(M\setminus L)$.
\end{itemize}
\end{itemize}
Then (iii) $\Rightarrow$ (ii) $\Rightarrow$ (i) holds. Moreover, (i) $\Rightarrow$ (iii) holds $C^\infty$-generically.
\end{theorem}

\section{Proof of Theorem~\ref{main3}}\label{sec_thm_asymp_cycles}

The main input in the proof is the following statement proved with pseudo-holomorphic curves.

\begin{theorem}[\cite{hryn_jsg}]\label{thm_jsg}
Let $\lambda$ be any dynamically convex contact form on $(S^3,\xi_0)$. Then a periodic Reeb orbit bounds a disk-like global surface of section if, and only if, it is unknotted and has self-linking number $-1$.
\end{theorem}

Here there are no hidden genericity assumptions, the only assumption is that of dynamical convexity. A disk-like global surface of section $D$ spanned by some unknotted, self-linking number $-1$ periodic orbit $\gamma=\partial D$ obtained from the above result has the following property: the first return time
\begin{equation*}
\tau : D\setminus\gamma \to (0,+\infty) \qquad\qquad \tau(p) \ = \ \inf \ \{ t>0 \mid \phi^t(p) \in D \}
\end{equation*}
is bounded, i.e.
\begin{equation}\label{upper_bound_return_time}
\sup_{p\in D\setminus\gamma} \tau(p) \ < +\infty.
\end{equation}
Since $D$ is a global surface of section, it follows from~\eqref{upper_bound_return_time} that there exists $L>0$ such that $\phi^{[0,L]}(q) \cap D \neq\emptyset$ for every $q \in S^3\setminus\gamma$.

Let $\gamma_1,\dots,\gamma_N$ be a collection of unknotted, self-linking number $-1$ periodic Reeb orbits. These orbits are taken as knots, i.e. primitive orbits, oriented by the flow. Consider a disk-like global surface of section $D_i$ spanned by $\gamma_i$, provided by Theorem~\ref{thm_jsg}, oriented in such a way that the identity $\partial D_i = \gamma_i$  takes orientations into account. Algebraically counting intersections with $D_i$ induces a cohomology class $y_i \in H^1(S^3\setminus\gamma_i;\R)$. Denoting inclusion maps by $\iota_j:S^3 \setminus \cup_i\gamma_i \to S^3\setminus\gamma_j$ we get a cohomology class
\begin{equation}\label{class_y}
y = \sum_i \iota_i^*y_i \ \in \ H^1(S^3\setminus\cup_i\gamma_i;\R).
\end{equation}
Denote also
\begin{equation}\label{l_ij}
\ell_{ij} = \link(\gamma_i,\gamma_j) \geq 1
\end{equation}
which are positive integers since all $D_i$ are global surfaces of section.

Let $T_i$ denote the primitive period of $\gamma_i$. With $i$ fixed consider a small smooth compact neighbourhood $\mathcal{N}_i$ and a smooth, orientation preserving, diffeomorphism $\Psi_i:\mathcal{N}_i \to \R/T_i\Z \times \D$ such that $\Psi_i\circ \phi^t\circ \Psi_i^{-1}(0,0)=(t,0)$. Here $\D \subset \C$ denotes the unit disk oriented by the complex orientation. Up to twisting, we may assume that $\Psi_i$ is aligned with $D_i$, i.e. if $\epsilon>0$ is small then the linking number of the loop $t\mapsto \Psi_i^{-1}(t,\epsilon)$ with $\gamma_i$ is equal to zero. Denote by $re^{i\theta}$ the polar coordinates on $D_i$. It follows that with respect to the basis $\{dt/T_i,d\theta/2\pi\}$ of $\R/T_i\Z \times (\D\setminus\{0\})$ we can write
\begin{equation*}
(\Psi_i)_*y = \left( \sum_{j\neq i} \ell_{ij} \right) \frac{dt}{T_i} + \frac{d\theta}{2\pi}
\end{equation*}
It follows from this and from the definition of the rotation number~\eqref{def_rot_number} that
\begin{equation}
\begin{aligned}
2\pi \rho^y(\gamma_i) &= T_i \left( \sum_{j\neq i} \frac{\ell_{ij}}{T_i} + \frac{1}{2\pi} \lim_{t\to+\infty} \frac{\theta(t)}{t} \right) \\
&= \sum_{j\neq i} \ell_{ij} + \lim_{t\to+\infty} \frac{\theta(t)/2\pi}{t/T_i} \\
&\geq \lim_{t\to+\infty} \frac{\theta(t)/2\pi}{t/T_i}
\end{aligned}
\end{equation}
where~\eqref{l_ij} was used in the third line. We claim that this limit is strictly positive. This will follow from $\CZ(\gamma_i)\geq 3$ together with ${\rm sl}(\gamma_i)=-1$. Here we write $\CZ$ for the Conley-Zehnder index in a global $d\lambda$-symplectic frame of $(\xi_0,d\lambda)$. In fact, the global $d\lambda$-symplectic frame of $\xi_0$ rotates ${\rm sl}(\gamma_i)=-1$ turns with respect to a $d\lambda$-symplectic of $\xi|_{\gamma_i}$ aligned with $D_i$. It turns out that there exists $\alpha_i \in \R$ such that $\CZ(\gamma_i^k) = 2 \lfloor k\alpha_i \rfloor + p(\gamma_i^k)$ for every $k\geq1$, where $|p(\gamma_i^k)| \leq 1$, and that if $\CZ(\gamma_i) \geq 3$ then $\alpha_i>1$. Hence 
\begin{equation*}
\lim_{t\to+\infty} \frac{\theta(t)/2\pi}{t/T_i} = \lim_{k\to+\infty} \frac{\CZ^{D_i}(\gamma_i^k)}{2k} = \lim_{k\to+\infty} \frac{\CZ(\gamma_i^k)-2k}{2k} = \alpha_i - 1 > 0.
\end{equation*}
Hence we are done checking
\begin{equation}
\rho^y(\gamma_i) > 0 \qquad \forall i
\end{equation}
which is (iii-a) in Theorem~\ref{thm_SFS}.

Now we check (iii-b). Let $\mu \in \Pscr_\phi(S^3\setminus\cup_i\gamma_i)$ be arbitrary. As explained in subsection~\ref{ssec_asymp_cycles}, there exists a Borel set $E \subset M\setminus\cup_i\gamma_i$ contained in the set of recurrent points such that $\mu(E)=1$, and for all $p\in E$ the limits $\lim_{n\to\infty} \left<y,k(T_n,p)\right>/t_n$ exist independently of the sequence $t_n\to +\infty$ satisfying $\phi^{t_n}(p) \to p$ and define a function $f_{\mu,y} \in L^1(\mu)$ whose integral is $\mu\cdot y$. Since each $D_i\setminus\gamma_i$ is transverse to the flow we conclude that $\mu(E\setminus \cup_iD_i) = 1$. Fix $p\in E\setminus \cup_iD_i$ and a sequence $t_n\to +\infty$ satisfying $\phi^{t_n}(p) \to p$. Then using the (positive) transversality of the flow with all the surfaces $D_i\setminus\gamma_i$
\begin{equation}\label{est_number_int_pts_1}
n \gg 1 \qquad \Rightarrow \qquad \left<y,k(T_n,p)\right> = \sum_i \# \{t\in[0,T_n] \mid \phi^t(p) \in D_i \}
\end{equation}
But
\begin{equation}\label{est_number_int_pts_2}
\# \{t\in[0,T_n] \mid \phi^t(p) \in D_i \} \geq \frac{T_n}{\sup \ \tau_i} - 1
\end{equation}
where $\tau_i$ is the return time function of $D_i$. Recall that $\sup \tau_i < +\infty$~\eqref{upper_bound_return_time}. Plugging~\eqref{est_number_int_pts_2} into~\eqref{est_number_int_pts_1} we obtain
\begin{equation}
\frac{\left<y,k(T_n,p)\right>}{T_n} \geq \sum_i \left( \frac{1}{\sup \ \tau_i} - \frac{1}{T_n} \right)
\end{equation}
Taking the limit as $n\to\infty$
\begin{equation}
f_{\mu,y} \geq \sum_i \frac{1}{\sup \ \tau_i} \ \text{($\mu$-almost everywhere)} \qquad \Rightarrow \qquad \mu\cdot y \geq \sum_i \frac{1}{\sup \ \tau_i}>0
\end{equation}
and we are done checking (iii-b). A direct application of Theorem~\ref{thm_SFS} concludes the proof of Theorem~\ref{main3}.

\section{Proof of Theorem~\ref{main4}}\label{sec_fast_planes}

Let $P=(x,T=m_0T_0)$ be a periodic Reeb orbit, with multiplicity $m_0$, of a defining contact form $\lambda$ on $(S^3,\xi_0)$. Here $T_0$ denotes the primitive period of~$x$. Throughout this section we denote by $\tau$ a global $d\lambda$-symplectic trivialisation of $\xi_0$. Assume that $\lambda$ is dynamically convex up to action $T$, and also that $\lambda$ is nondegenerate up to action $T$.

\begin{proposition}\label{prop_SFS_input}
If $\util$ is a fast plane asymptotic to $P$ then there exists $a>0$ such that 
\begin{equation}
\#\{t\in[0,T] \mid \phi^t(p) \in u(\C)\} \geq \left\lfloor \frac{T}{a} \right\rfloor
\end{equation}
holds for every $p\in S^3 \setminus x(\R)$ and every $T\geq0$.
\end{proposition}

We first show that Proposition~\ref{prop_SFS_input} can be used to check the hypothesis of Theorem~\ref{thm_SFS} for the periodic orbit $x(\R)$ and the cohomology class counting linking numbers with it. Theorem~\ref{main4} follows as a consequence.

\begin{proof}[Proof that Theorem~\ref{main4} follows from Proposition~\ref{prop_SFS_input}]
Let $y\in H^1(S^3\setminus x(\R);\R)$ be the cohomology class that counts linking number of loops in $S^3\setminus x(\R)$ with the loop $t \in \R/\Z \mapsto x(T_0t)=x(Tt/m_0)$. Here we ignore $\Z$-coefficients and work with $\R$-coefficients. If we compactify $\C$ to a disk $D_\C$ by adding a circle at $\infty$ then $u$ induces a capping disk $\bar u:D_\C\to S^3$ for $P$ such that the class in $H^1(S^3\setminus x(\R))$ dual to $\bar u_*[D_\C] \in H_2(S^3,x(\R))$ is precisely~$m_0y$. Here $[D_\C]$ is the fundamental class in $H_2(D_\C,\partial D_\C;\Z)$ induced by the complex orientation. Observe that $u(\C) \setminus x(\R)$ has measure zero with respect to any $\mu\in\Pscr_\phi(S^3\setminus x(\R))$ since it is transverse to the flow. Hence, in view of the discussion in subsection~\ref{ssec_asymp_cycles}, we get a Borel set $E \subset S^3\setminus x(\R)$ such that $\mu(E)=1$ and every point $p\in E$ has the following properties:
\begin{itemize}
\item[(a)] $p$ is recurrent.
\item[(b)] The limits $$ \lim_{n\to+\infty} \frac{\link(k(T_n,p),x(\R))}{T_n} = \lim_{n\to+\infty} \frac{\left<y,k(T_n,p)\right>}{T_n} $$ exist independently of $T_n \to +\infty$ satisfying $\phi^{T_n}(p) \to p$ (and of auxiliary Riemannian metrics), and define a function in $L^1(\mu)$ whose integral is equal to the intersection number $\mu\cdot y$.
\item[(c)] $p\not\in u(\C)$.
\end{itemize}
Hence, using the transversality between $u$ and the Reeb vector field, for every $p \in E$ we can estimate 
\begin{equation}\label{linking_first_est}
\begin{aligned}
m_0 \ \link(k(T_n,p),x(\R)) &= \left<m_0y,k(T_n,p)\right> \\
&=\sum_{t\in[0,T_n], \ \phi^t(p) \in u(\C)} \#\{z\in\C \mid u(z) = \phi^t(p)\} \\
&\geq \#\{t\in[0,T_n] \mid \phi^t(p) \in u(\C)\}
\end{aligned}
\end{equation}
for all $n$ large enough, where $T_n\to+\infty$ satisfies $\phi^{T_n}(p) \to p$. With the aid of Proposition~\ref{prop_SFS_input} we can estimate from~\eqref{linking_first_est}
\begin{equation}
\lim_{n\to+\infty} \frac{\link(k(T_n,p),x(\R))}{T_n} \geq \lim_{n\to+\infty} \frac{1}{m_0T_n} \left\lfloor \frac{T_n}{a} \right\rfloor = \frac{1}{m_0a} \qquad \forall p\in E
\end{equation}
which implies, by definition of intersection numbers, that
\begin{equation}
\mu\cdot y \geq \frac{1}{m_0a}>0 \qquad \forall \mu \in \Pscr_\phi(S^3\setminus x(\R))
\end{equation}
Condition $\rho^y(x(\R)) > 0$ follows immediately from $\CZ(P_0) \geq 3$ where $P_0=(x,T_0)$ is the simply covered periodic orbit underlying $P$. 
Theorem~\ref{main4} now follows from a direct application of Theorem~\ref{thm_SFS} since $y$ is dual to %counts the intersection number of any loop in the complement of $x(\R)$ with 
any Seifert surface for $x(\R)$; here it was used that the ambient space is $S^3$.
\end{proof}

To complete the proof of Theorem~\ref{main4} we need to establish Proposition~\ref{prop_SFS_input}. The rest of this section is concerned with the proof of Proposition~\ref{prop_SFS_input}.

Let us denote by $P_0=(x,T_0)$ the simply covered periodic orbit underlying $P$, and recall that $m_0$ denotes the covering multiplicity of $P=(x,T=m_0T_0)$. For every $k\geq1$ we denote $P_0^k = (x,kT_0)$. In particular $P = P_0^{m_0}$.

Consider the set $\Mfast(P,J)$ of equivalence classes of fast finite-energy planes asymptotic to $P$, where two planes $\util,\vtil$ are equivalent if there exist $A\in\C^*$, $B\in\C$ and $c\in\R$ such that $\vtil_c(z) = \util(Az+B)$ holds for every $z\in\C$. Here $\vtil_c$ denotes the translation of $\vtil$ by $c$ in the $\R$-component. Equivalence classes are denoted by~$[\util]$.

It is possible to build a Fredholm theory for $\Mfast(P,\jbar)$ modelled on sections of the normal bundle, using Sobolev or H\"older spaces. Fix a number $\delta<0$ in the spectral gap of the asymptotic operator associated to $(P,J)$ between eigenvalues with winding number $1$ and $2$ with respect to $\tau$. This is possible since $\CZ_\tau^0(P)\geq 3$. Note that $\alpha^{<\delta}_\tau(P) = 1$ and $\alpha^{\geq\delta}_\tau(P)=2$. Let
\begin{equation}
\util=(a,u):(\C,i) \to (\R\times S^3,\jtil)
\end{equation}
be a fast plane representing an element of $\Mfast(P,J)$. Consider the space of sections of the normal bundle of $\util(\C)$ with exponential decay faster than $\delta$. The Fredholm index of the linearisation $D_{\util}$ of the Cauchy-Riemann equations at $\util$ restricted to this space of sections is
\begin{equation}
{\rm ind}_\delta(\util) = \CZ^\delta_\tau(P) - 1 = 3-1 = 2.
\end{equation}

An important fact is that \textit{automatic transversality} holds, i.e. $D_{\util}$ at a fast plane~$\util$ is always a surjective Fredholm operator. Let us prove this fact. There is no loss of generality to deform the normal bundle so that it coincides with $u^*\xi_0$ over $\C\setminus B_R(0)$, $R\gg1$. A $d\lambda$-symplectic trivialising frame of the normal bundle induces, up to homotopy, a $d\lambda$-symplectic trivialisation $\sigma_N$ of $x(T\cdot)^*\xi_0$ which winds $+1$ with respect to the global frame $\tau$. Moreover, a nontrivial section $\zeta\in\ker D_{\util}$ admits an asymptotic behaviour governed by an eigensection of the asymptotic operator associated to an eigenvalue $\nu<\delta$, see~\cite[Theorem~6.1]{fast} or~\cite[Theorem~A.1]{sie_CPAM}. Hence, $\zeta$ does not vanish near $\infty$ and the total algebraic count of zeros of $\zeta$ is equal to the winding number of $\nu$ with respect to $\sigma_N$, which is equal to $$ \wind_\tau(\nu)-1 \leq \alpha^{<\delta}_\tau(P)-1=1-1=0. $$ But the equation $D_{\util}\zeta=0$ allows us to use Carleman's similarity principle to say that zeros are isolated and count positively. The important conclusion is that $\zeta$ never vanishes. Since the Fredholm index is $2$, we would find $3$ linearly independent sections of the kernel if $D_{\util}$ were not surjective. But the normal bundle is two-dimensional, hence a nontrivial linear combination of them would have to vanish at some point, contradiction.

\begin{remark}
Arguments like the one used above to prove automatic transversality statements were explored in~\cite{HLS,Wendl}, see also~\cite{tese}.
\end{remark}

It follows from the above discussed automatic transversality that $\Mfast(P,J)$ can be given the structure of a smooth, Hausdorff and second countable one dimensional manifold.

\begin{remark}
Under our standing assumption that $\lambda$ is non-degenerate up to action~$T$ one can show that the topology on $\Mfast(P,J)$ inherited from the functional analytic set-up used for the Fredholm theory coincides with the topology of $C^\infty_{\rm loc}$-convergence. There are situations where this can also be proved dropping non-degeneracy~\cite{hryn_jsg,elliptic}.
\end{remark}

Consider a sequence $\util_n:(\C,i) \to (\R\times S^3,\jtil)$ of fast finite-energy planes asymptotic to $P$. Since $\lambda$ is assumed to be nondegenerate up to action $T$ we can apply the SFT compactness theorem to get, up to selection of a subsequence, that $\util_n$ SFT-converges to a stable holomorphic building~$\mathbf{u}$. Since~${\bf u}$ is a limit of planes it can be conveniently described as a directed, rooted tree~$\mathcal{T}$. Each vertex $v$ corresponds to a finite-energy map
\begin{equation*}
\util_v = (a_v,u_v) : (\C\setminus\Gamma_v,i) \to (\R\times S^3,\jtil)
\end{equation*}
with a unique positive puncture $\infty$. The finite set $\Gamma_v$ consists of the negative punctures of $\util_v$. The top level of this building corresponds to the root $r$, and consists of a single finite-energy map $\util_r$ which is asymptotic to $P$ at its positive puncture $\infty$. Edges are always assumed oriented as going away from the root. An edge $e$ from the vertex $v$ to the vertex $v'$ corresponds to a negative puncture of $\util_{v}$. The asymptotic limit $\util_{v}$ at the negative puncture corresponding to $e$ is equal to the asymptotic limit of $\util_{v'}$ at its positive puncture. The leaves correspond precisely to the vertices $v$ such that $\util_v$ is a plane ($\Gamma_v = \emptyset$).

%The top level of this building must be represented by a finite-energy map
%\begin{equation}
%\util = (a,u) : \C\setminus\Gamma \to \R\times S^3
%\end{equation}
%with one positive puncture at $\infty$ where it is asymptotic to $P$. The set $\Gamma$ consists of negative punctures. Moreover, $\int u^*d\lambda>0$, i.e. $\util$ is not a (possibly branched) cover of a trivial cylinder; otherwise stability of $\mathbf{u}$ is violated.

\begin{lemma}\label{lemma_limits_of_fast_planes}
If $v$ is a vertex of $\mathcal{T}$ such that $\int u_v^*d\lambda > 0$ then $\wind_\infty(\util_v,\infty,\tau)\leq 1$.
\end{lemma}

\begin{proof}
SFT compactness allows us to find $A_n \in \C^*$, $B_n \in \C$ and $c_n \in \R$ such that the planes $\wtil_n(z) = c_n \cdot \util_n(A_nz+B_n)$ converge to $\util_v$ in $C^\infty_{\rm loc}(\C\setminus\Gamma_v)$. Here $c_n \cdot \util_n$ denotes the translation by $c_n$ in the $\R$-component.

Consider components $\wtil_n=(d_n,w_n)$ and $\util_v=(a_v,u_v)$ in $\R\times S^3$. Write $\wtil_n(s,t)=(d_n(s,t),w_n(s,t))$ instead of $\wtil_n(e^{2\pi(s+it)})$, and similarly $\util_v(s,t)=(a_v(s,t),u_v(s,t))$. Fix $s_0$ such that $z\in\Gamma_v \Rightarrow |z| < e^{2\pi s_0}$. 
%Let $\mathcal{N}$ be a small tubular neighbourhood of $x(\R)$ in $S^3$. 
By Theorem~\ref{thm_asymptotics} we can find $s_1>s_0$ such that $\pi_\lambda(\partial_su_v)$ does not vanish on $[s_1,+\infty)\times\R/\Z$ and the winding number $\wind(\pi_\lambda(\partial_su_v)(s_1,\cdot))$ of $t\mapsto \pi_\lambda(\partial_su_v)(s_1,t)$ in the global frame $\tau$ is equal to $\wind_\infty(\util_v,\infty,\tau)$. Since $\pi_\lambda(\partial_sw_n) \to \pi_\lambda(\partial_su_v)$ in $C^\infty_{\rm loc}$ we find $n_0$ such that if $n\geq n_0$ then $\pi_\lambda(\partial_sw_n)$ does not vanish on $\{s_1\}\times\R/\Z$ and 
$$ 
\wind(\pi_\lambda(\partial_sw_n)(s_1,\cdot)) = \wind(\pi_\lambda(\partial_su_v)(s_1,\cdot)) = \wind_\infty(\util_v,\infty,\tau). 
$$ 
The frame $\tau$ can be used to represent the maps $(s,t) \mapsto \pi_\lambda(\partial_sw_n)$ by smooth maps $\zeta_n : [s_0,+\infty) \times\R/\Z \to \C$ satisfying a Cauchy-Riemann type equation. Carleman's similarity principle implies that either $\zeta_n$ vanishes identically on $[s_0,+\infty)\times\R/\Z$, or its zeros are isolated and count positively. It can not vanish identically since the $\wtil_n$ are planes. By Theorem~\ref{thm_asymptotics} $\zeta_n(s,t)$ does not vanish when $s$ is large enough and for every $n$ we have $$ \lim_{s\to+\infty} \wind(\zeta_n(s,\cdot)) = \lim_{s\to+\infty} \wind(\pi_\lambda(\partial_sw_n)(s,\cdot)) = \wind_\infty(\wtil_n) = \wind_\infty(\util_n) $$ If $s>s_1$ is large enough then $\wind(\zeta_n(s,\cdot)) - \wind(\zeta_n(s_1,\cdot))$ is the algebraic count of zeros of $\zeta_n$ on $[s_1,s]\times\R/\Z$. Since this count is nonnegative we get $$ \wind(\pi_\lambda(\partial_sw_n)(s_1,\cdot)) \leq \wind_\infty(\util_n) $$ for all $n\geq n_0$. Hence 
$$ 
n\geq n_0 \Rightarrow \wind_\infty(\util_v,\infty,\tau) \leq \wind_\infty(\util_n) = 1
$$ 
as desired. %The conclusion follows from this inequality together with~\eqref{def_wind_pi} and~\eqref{identity_winds}.
\end{proof}

\begin{lemma}
If the vertex $v$ is not a leaf then $\int u_v^*d\lambda=0$, i.e. $\util_v$ is a possibly branched cover of a trivial cylinder over a periodic orbit.
\end{lemma}

\begin{proof}
Suppose that $\int u_v^*d\lambda>0$. At the negative punctures $z\in\Gamma_v$ of $\util_v$ we have $\wind_\infty(\util_v,z,\tau) \geq 2$ since the asymptotic limits at these punctures are periodic Reeb orbits with action less than $T$ and hence, by assumption, satisfy $\CZ^0_\tau\geq3$. By the previous lemma together with~\eqref{def_wind_pi} and~\eqref{identity_winds} we arrive at
\begin{equation}
\begin{aligned}
0 &\leq \wind_\pi(\util_v) = \wind_\infty(\util_v) - 1 + \#\Gamma_v \leq 1 - 2\#\Gamma_v -1 + \#\Gamma_v = -\#\Gamma_v
\end{aligned}
\end{equation}
Thus $\Gamma_v=\emptyset$ and $v$ is a leaf.
\end{proof}

\begin{corollary}\label{cor_structure_of_tree}
The following dichotomy holds for every vertex $v$ of $\mathcal{T}$:
\begin{itemize}
\item[(i)] $v$ is not a leaf, $\int u_v^*d\lambda=0$ and $\util_v$ is a (possibly branched) cover of a trivial cylinder.
\item[(ii)] $v$ is a leaf, $\int u_v^*d\lambda>0$ and $\util_v$ is a fast plane asymptotic to a covering of~$P_0$.
\end{itemize}
\end{corollary}

\begin{proof}
Case (i) is handled by the previous lemma. Let us now argue for (ii). By Lemma~\ref{lemma_limits_of_fast_planes} if $v$ is a leaf then it is a plane satisfying $\wind_\infty(\util_v)\leq1$. Hence $0\leq\wind_\pi(\util_v) = \wind_\infty(\util_v)-1 \leq1-1=0$, i.e. $\wind_\infty(\util_v)=1$ and $\util_v$ is a fast plane.
\end{proof}

For every $1\leq k\leq m_0$ we consider $\Mfast(P_0^k,J)$ the moduli space of fast finite-energy planes asymptotic to $P_0^k$, defined as before. For each $k$ there is a suitable choice of negative weight placed precisely at the spectral gap between eigenvalues of the asymptotic operator associated to $(P_0^k,J)$ with winding $1$ and $2$ in a global frame. With these weights one builds a Fredholm theory as before, and there is automatic transversality. The space $\Mfast(P_0^k,J)$ becomes a $1$-dimensional smooth, second countable Hausdorff manifold. Moreover, the induced topology coincides with the topology induced by $C^\infty_{\rm loc}$-convergence.

\begin{corollary}\label{c:fast-planes-S1-family}
There exists $m \in \{1,\dots,m_0\}$ such that $\Mfast(P_0^m,J)$ is non-empty and compact.
\end{corollary}

\begin{proof}
We work under the assumption of Theorem~\ref{main4} that $\Mfast(P=P_0^{m_0},J)$ is non-empty. If $\Mfast(P,J)$ is compact then there is nothing to be proved. If $\Mfast(P,J)$ is not compact then some sequence in $\Mfast(P,J)$, represented by fast planes $\util_n$, will SFT-converge to a building ${\bf u}$ with more than one level. This means that the corresponding tree $\mathcal{T}$ does not consist of a single vertex (the root), and by Corollary~\ref{cor_structure_of_tree} every leaf $v$ must be a fast finite-energy plane asymptotic to $P_0^{k_v}$, for some $k_v \in \{1,\dots,m_0-1\}$. The reason for the strict inequality $k_v<m_0$ is that the root must have at least two negative punctures: otherwise the root corresponds to a trivial cylinder, which is ruled out by stability of the limiting building. Pick any leaf~$v$, denote $m_1 = k_v$. Hence the moduli space $\Mfast(P_0^{m_1},J)$ of fast planes asymptotic to $P_0^{m_1}$ is not empty. If $\Mfast(P_0^{m_1},J)$ is compact then we are done with the proof. If not we proceed just as above to find $1\leq m_2\leq m_1-1$ such that $\Mfast(P_0^{m_2},J)$ is non empty. After a finite number of steps $k\geq0$ this process stops and we find $1\leq m_k \leq m_0$ such that $\Mfast(P_0^{m_k},J)$ is non-empty and compact.
\end{proof}

%\begin{lemma}
%The building $\mathbf{u}$ has one level, i.e. the tree $\mathcal{T}$ consists precisely of the root $r$, and $\util_r$ is a fast plane asymptotic to $P$.
%\end{lemma}
%
%\begin{proof}
%If ${\bf u}$ has more than one level then $r$ is not a leaf, so the previous lemma implies that $\int u_r^*d\lambda=0$. But this implies that $\util_r$ is a trivial cylinder because $P$ is simply covered by assumption, contradicting the stability of the building ${\bf u}$. Hence $r$ is the only vertex, and Lemma~\ref{lemma_limits_of_fast_planes} implies that it is a fast plane asymptotic to~$P$.
%\end{proof}

From now on $m$ is given by the previous lemma, that is, $\Mfast(P_0^m,J)$ is a non-empty, compact, smooth and Hausdorff $1$-dimensional manifold, i.e. a finite collection of circles. %Note that the $(\R,+)$ action on $\Mfast(P_0^m,J)$ is smooth and free. 

%This fact and the previous discussion imply that following statement. 
%
%
%\begin{proposition}\label{prop_compactness}
%If $P=(x,T)$ is a periodic Reeb orbit of a contact form $\lambda$ on $(S^3,\xi_0)$ which is both nondegenerate and dynamically convex up to action $T$, then $\Mfast(P,J)/\R$ is a compact smooth one dimensional manifold.
%\end{proposition}

Consider the space $\Mfast_1(P_0^m,J)$ of equivalence classes of pairs $(\util,z)$ where $\util$ is a fast plane asymptotic to $P_0^m$ and $z\in\C$. Two pairs $(\util_0,z_0)$, $(\util_1,z_1)$ are equivalent if there exist $A\in\C^*$, $B\in\C$ such that $\util_1(Az+B) = \util_0(z)$ for all $z\in\C$ and $z_1=Az_0+B$. Note that $(\R,+)$ acts freely on $\Mfast_1(P_0^m,J)$ by translations in the symplectisation direction. Hence $\Mfast_1(P_0^m,J)/\R$ is a smooth three-dimensional manifold. The map
\begin{equation}\label{ev_map}
\ev:\Mfast_1(P_0^m,J)/\R \to S^3 \qquad\qquad \ev([\util=(a,u),z]/\R) \mapsto u(z)
\end{equation}
is smooth.

\begin{lemma}\label{lemma_ev_submersion}
The map $\ev$ is a submersion.
\end{lemma}

\begin{proof}
For every $\util \in \Mfast(P_0^m,J)$ nontrivial sections in the kernel of the linearised Cauchy-Riemann operator at $\util$, with the appropriate weighted exponential decay, which represent elements in the tangent space, never vanish and $u$ is an immersion.
\end{proof}

\begin{lemma}\label{lemma_ev_relatively_proper}
If $K \subset S^3\setminus x(\R)$ is compact then $\ev^{-1}(K)$ is compact.
\end{lemma}

\begin{proof}
%We already know from Proposition~\ref{prop_compactness} that $\Mfast(P,J)/\R$ is compact. 
Suppose that $[\util_n,z_n]$ represents a sequence in $\ev^{-1}(K)$. Up to reparametrisation, translation in the $\R$-component, and selection of a subsequence, we may assume that $\util_n$ converges in $C^\infty_{\rm loc}$ to some plane $\util$ representing an element of $\Mfast(P_0^m,J)$. Let $\mathcal{N}$ be a neighbourhood of $x(\R)$ such that $K \cap \mathcal{N} = \emptyset$. One can then invoke results on cylinders of small contact area from~\cite{small_area} to conclude that there exists $R$ and $n_0$ such that if $n\geq n_0$ and $|z|\geq R$ then $\util_n(z) \in \R\times \mathcal{N}$. This implies that $\sup_n|z_n|\leq R$. Hence one can assume, up to selection of a subsequence, that $z_n \to z$ for some $z$. It follows that $[\util_n,z_n]/\R \to [\util,z]/\R$. 
\end{proof}

\begin{lemma}
The image of the map $\ev$ contains $S^3\setminus x(\R)$.
\end{lemma}

\begin{proof}
By Lemma~\ref{lemma_ev_submersion} the image is open in $S^3$, hence its intersection with $S^3\setminus x(\R)$ is an open subset of $S^3\setminus x(\R)$. By Lemma~\ref{lemma_ev_relatively_proper} the intersection of the image of $\ev$ with $S^3\setminus x(\R)$ is a closed subset of $S^3\setminus x(\R)$. The conclusion follows from connectedness of $S^3\setminus x(\R)$.
\end{proof}

Consider $[\util=(a,u)] \in \Mfast(P_0^m,J)$ and the function
\begin{equation}
\tau : S^3\setminus x(\R) \to [0,+\infty] 
\end{equation}
defined by
\begin{equation}
\tau(p) \ = \ \inf \{ t>0 \mid \phi^t(p) \in u(\C) \}
\end{equation}
with the convention that the infimum of the empty set is $+\infty$.

\begin{lemma}\label{lemma_bounds_on_return_times}
$\tau$ takes values on $(0,+\infty)$, and $\sup \tau < +\infty$.
\end{lemma}

\begin{proof}
From the transversality of $u$ to the Reeb flow, and the asymptotic formula from Theorem~\ref{thm_asymptotics}, we conclude that given any $[\vtil=(b,v)] \in \Mfast(P_0^m,J)$ and $p\in S^3$, the set $v^{-1}(p)\subset \C$ is finite, and also that $\tau$ takes values on $(0,+\infty]$.  

Suppose that $p\not\in x(\R)$ and $\omega(p) \cap x(\R) \neq \emptyset$. By invariance of $x(\R)$ under the Reeb flow, the trajectory $\phi^t(p)$ will spend arbitrarily long times in the future arbitrarily and uniformly close to $x(\R)$. Hence, the way in which it rotates around $x(\R)$ is governed by the linearised Reeb flow along $x$. Every plane $\vtil=(b,v)$ representing an element in $\Mfast(P_0^m,J)$ is asymptotic to $P_0^m$ according to an eigenvector of a negative eigenvalue of the asymptotic operator with winding $+1$ in a global frame; this information is encoded in $\wind_\infty(\vtil)=1$. Hence, in transverse polar coordinates aligned with the global frame the plane rotates $2\pi$. After one period $T=mT_0$ the linearised flow rotates every transverse vector by an angle larger than $2\pi+\Delta$ for some uniform $\Delta>0$. This information is encoded in $\CZ^0_\tau(P_0^m)\geq 3$. Hence after flow time of about $\lfloor\frac{2\pi}{\Delta}+1\rfloor T$ any point nearby $P_0$ already returned once back to the plane. It follows that the return time is bounded from above for points near $P_0$.
%By the asymptotic formula, trajectories near $x$ will hit $v(\C)$ at least once after flow time of about the period $mT_0$ of $P_0^m$. This argument actually shows that for $p$ close enough to $x(\R)$ we get an estimate $\tau(p) \leq 2mT_0$.

If $\omega(p) \cap x(\R) = \emptyset$ then it follows from compactness of $\Mfast(P_0^m,J)$ and transversality of the planes to the Reeb flow that for every $[\vtil] \in \Mfast(P_0^m,J)$ the trajectory $\phi^t(p)$ will hit $v(\C)$ in finite time.

So far we have proved that $\tau$ takes values on $(0,+\infty)$, and that $\tau$ is bounded near $x(\R)$. To conclude we note that if $$ \{w_1,\dots,w_N\} = u^{-1}(\phi^{\tau(p)}(p)) $$ then there are $N$ local smooth germs of hitting times $\tau_1,\dots,\tau_N$ near $p$. Then $\tau$ can be locally bounded in terms of these germs.
\end{proof}

Proposition~\ref{prop_SFS_input} is a consequence of Lemma~\ref{lemma_bounds_on_return_times}.

\begin{remark}
We observe that the finite energy planes produced by Corollary \ref{c:fast-planes-S1-family}
can themselves be thought of as sorts of generalized surfaces of section where we allow for the possibility that
the surface is an immersion rather than embedding.  Indeed our proof shows that
the projection to $S^{3}$ of 
every such plane is 
an immersion, transverse to the Reeb flow, and that the flow line through any given point in
$S^{3}\setminus P_{0}$ will hit the surface in forward and backward time.
In the case that the plane is not an embedding,
it follows from results in \cite{props2, sief_int} that
it must intersect its asymptotic limit, and thus in this case the plane will intersect
the flow line through any given point in $S^{3}$ including points in $P_{0}$.

We observe further that, since our proof shows that the evaluation map
\[
\ev:\Mfast_1(P_0^m,J)/\R \to S^3
\]
is an immersion between manifolds of the same dimension, it is also a local diffeomorphism, so we can 
use $\ev^{-1}$ to lift the flow to the moduli space $\Mfast_1(P_0^m,J)/\R$,
each component of which is diffeomorphic to $\C\times S^{1}$.
Moreover, since each plane in $\Mfast_1(P_0^m,J)/\R$
is transverse to the flow, the resulting flow on $\Mfast_1(P_0^m,J)/\R$ will be transverse to the
disk-like fibers
of the forgetful map $\Mfast_1(P_0^m,J)/\R\mapsto\Mfast(P_0^m,J)/\R$.
So although the surface of section provided by our theorem will in general have genus,
the fast finite energy planes
that we construct in the proof can themselves be used to visualize the dynamics as a return map 
on a disk.
\end{remark}

\end{document}